\documentclass{amsart}

\numberwithin{equation}{section}

\newtheorem{lemma}[equation]{Lemma}
\theoremstyle{definition}
\newtheorem{example}[equation]{Example}
\newtheorem{method}[equation]{Method}

\newcommand{\cref}[1]{Corollary~\textup{\ref{#1}}}
\newcommand{\lref}[1]{Lemma~\textup{\ref{#1}}}
\newcommand{\eref}[1]{Example~\textup{\ref{#1}}}
\newcommand{\mref}[1]{Method~\textup{\ref{#1}}}

\newcommand{\C}{\mathbb C}

\newcommand{\rank}{\operatorname{rank}} 

\def\new{\goodbreak\bigskip\noindent}

\def\J#1[#2]{\mathcal{J}_{#1\,}[#2]}

\def\B{\mathcal B}

\newcommand{\con}{\,\mathcal C}

\newcommand{\upitem}[1]{\textup{(#1)}}

\newcommand{\xx}{x_1,x_2,x_3}

\newcommand{\uu}{u_1,u_2,u_3}

\usepackage[outline]{contour}
\newcommand*{\fancy}[1]{{\color{white}\contour{black}{#1}}}
\newcommand{\bF}{\fancy{$F$}}
\newcommand{\bDelta}{\fancy{$\Delta$}}
\newcommand{\btheta}{\fancy{$\theta$}}
\newcommand{\bS}{\fancy{$S$}}
\newcommand{\bT}{\fancy{$T$}}

\makeatletter
\@namedef{subjclassname@2020}{\textup{2020} Mathematics Subject Classification}
\makeatother

\begin{document}

\title{The Waring Rank of Ternary Cubic Forms}
\author{Gary Brookfield}
\email{gbrookf@calstatela.edu}
\address{California State University, Los Angeles}
\begin{abstract}
We provide conditions on the coefficients of a ternary cubic form that determine its Waring rank. 
\end{abstract}

\maketitle

\section{Introduction}

Let \( f\in \C[ x_1, x_2,\dots, x_n] \) be a degree \( d \) form in variables \( x_1, x_2,\dots, x_n \) over the complex numbers. Then the \textbf{rank} of \( f \), written \(  \rank[f] \),  is the smallest natural number \( r \) such that \( f \) can be written as  a sum of \( r \) \( d^{\rm th} \) powers of linear forms in \( \C[ x_1, x_2,\dots, x_n] \).  To distinguish this definition of rank from others in the literature, \( \rank[f] \) is often called the \textbf{Waring rank} of the form \( f \).

For example, the cubic form \( f_{xxx}=x_1x_2x_3 \) can be written as a sum of four cubes:

\newcommand{\sq}[2][0]{
  \mbox{$\medmuskip=#1mu\displaystyle#2$}%
}

\begin{equation}\label{eq1}
\sq[1]{f_{xxx}=\frac1{24}\left((x_1+x_2+x_3)^3+(x_1-x_2-x_3)^3+(x_2-x_3-x_1)^3+(x_3-x_1-x_2)^3\right)}
\end{equation}
There are many other ways of expressing \( f_{xxx} \) as a sum of four cubes, but it is not possible to express \( f_{xxx} \) as a sum of three cubes. Hence the  rank of \( f_{xxx}=x_1x_2x_3 \)  is four.   See Section~\ref{sectreducible} for details.

The  coefficient \( 1/24 \) on the cubes in the above expression for \( f_{xxx} \) is immaterial. Because we are working over the complex numbers, any such coefficients can be absorbed into the linear forms being cubed.

The Waring rank of forms has been studied in many special cases over the past 150 years. For example, much is known about the  rank of binary forms, quadratic forms in any number of variables, completely reducible forms, monomials and binary binomials \cite{carlini2016,moncusi}. An introduction to the  rank problem and its history can be found in \cite{iarrobino1999}. Another general reference with an extensive bibliography is \cite{carlini2013}.

For ternary cubic forms, the state of the art is represented by the table of the  ranks of ternary cubic forms having particular normal forms  found in \cite[Table~1]{landsberg2010}.  See also \cite{comon1996}. But this situation is not very satisfactory in the sense that it is saying that if a cubic ternary form can be written in a certain way - one of eleven different normal forms - then it can also be written in another way as a sum of cubes. The purpose of this article, is to show how the  rank of a ternary cubic form can be determined directly from its coefficients.

\new It is fairly clear that the  rank of a form is unchanged by a linear change of the variables. For example, if \( f_{xxx}=a_xb_xc_x \) with \( \{a_x,b_x,c_x\} \)  linearly independent, then the linear transformation \( x_1\mapsto a_x \), \( x_2\mapsto b_x \), \( x_3\mapsto c_x \) sends \(  x_1x_2x_3 \) to \( f_{xxx} \). Applying this transformation to both sides of \eqref{eq1} we see that \( f_{xxx} \) is a sum of four cubes. See \eqref{eq128}. Since the linear transformation is invertible we get \( \rank[f_{xxx}]=\rank[x_1x_2x_3]=4 \). 

So, in general, the  rank of a form is an invariant of the form---it is unchanged by invertible linear transformations of its variables. Then it should be no surprise that the  rank of a form  depends on its concomitants (invariants, covariants, contravariants).

The simplest example of this is the generic case: If the  Aronhold invariant \( S \)  is nonzero for a ternary cubic form \( f_{xxx} \), then \( \rank[f_{xxx}]=4 \) (\lref{lem37}).  
This, and the other main results of this article, are summarized in the table in Section~\ref{sectmain}.

\section{Preliminaries}

All of the definitions and notation used in this article come from \textit{Completely Reducible Cubic Ternary Forms} \cite{brookfield2021}.  All forms have coefficients in \( \C \), the field of complex numbers, and all are ternary, that is, they are homogeneous polynomials in three variables \( \xx \).  We write  \( f_x \), \( f_{xx} \) and \( f_{xxx} \) for linear, quadratic and cubic forms in these variables. For a form of unspecified degree, we drop the subscripts. The \( n^{\text{th}} \) transvectant of forms  \( f \), \( g \) and \( h \) is defined in \cite[Section~4]{brookfield2021} and is written \( \J{n}[f,g,h] \). When nonzero, \( \J{n}[f,g,h] \) is a form of degree \( \deg f+\deg g+\deg h-3n \) in \( \xx \). The first transvectant is just the Jacobian, which we write as \( \J[f,g,h] \). Frequently we will use ``contravariant'' variables \( \uu \), which geometrically, are the variable coordinates of the line \( u_x=u_1x_1+u_2x_2+u_3x_3 \).

For forms \( a_{x}=a_1x_1+a_2x_2+a_3x_3 \),  \( b_{x}=b_1x_1+b_2x_2+b_3x_3 \) and \( c_{x}=c_1x_1+c_2x_2+c_3x_3 \), \( \J[a_x,b_x,c_x] \) is just the determinant of the coefficients of these forms.  For compactness, we abbreviate this as \( [abc] \): 
\[ [abc]=\J[a_x,b_x,c_x]=
\begin{vmatrix}
a_1  &a_2  & a_3 \\ b_1  &b_2  & b_3 \\c_1  &c_2  & c_3 
\end{vmatrix}
 \] 
So \( \{a_x,b_x,c_x\} \) is linearly dependent if and only if \( [abc]=0 \). A pair of  forms \( \{a_x,b_x\} \) is linearly dependent if and only if \( [abu]=0 \). Since \( \uu \) are variables, \( [abu] \) is a linear form in \( \uu \) with three coefficients that are functions of the coefficients of \( a_x \) and \( b_x \). Thus \( \{a_x,b_x\} \) is linearly dependent if and only if all three of these coefficients are zero.

We also need the contraction operator \( \con_{ux} \) defined by \[ \con_{ux}[f ]= \frac{\partial^2 f }{\partial x_1\partial u_1} +\frac{\partial^2 f }{\partial x_2\partial u_2}+\frac{\partial^2 f }{\partial x_3\partial u_3} \]
for all  forms \( f \) in \( \xx \) and \( \uu \). We use the notation \( \Vert \,\cdot\,\Vert_{\sigma}  \) for substitutions as explained in \cite{brookfield2021}. For example, we write  \( S \) for the  Aronhold invariant of the cubic form \( f_{xxx} \). The Aronhold invariant of another cubic form \( g_{xxx} \) could be written as \( \Vert S\Vert_{f\mapsto g} \) where \( f\mapsto g \) denotes the substitution of the coefficients of \( f_{xxx} \) by the coefficents of \( g_{xxx} \).

Somewhat imprecisely, we will say that the forms that can be constructed from a given form \( f \) and \( u_x \) using transvectants, contractions and substitutions are \textbf{concomitants} of \( f\). For the precise definition and for its importance to the theory of invariants, see Grace and Young, \textit{The Algebra of Invariants}  \cite{grace}.

\section{The  Rank of Quadratic Forms}

The solution of  rank problem for ternary quadratic forms is well known, but it is nonetheless worthwhile recording how the  rank of such a form \( f_{xx} \) is determined by the two concomitants \( \J2[f_{xx},f_{xx},u_x^2]=0 \)  and \( \J2[f_{xx},f_{xx},f_{xx}]=0 \).

\begin{lemma}\label{lem10}
Let \( f_{xx} \) be a quadratic form.
\begin{enumerate}
\item \( \rank[f_{xx}]=1 \) if and only if \( \J2[f_{xx},f_{xx},u_x^2]=0 \) and \( f_{xx}\neq 0 \).

\item \( \rank[f_{xx}]=2 \) if and only if \( \J2[f_{xx},f_{xx},f_{xx}]=0 \) and \( \J2[f_{xx},f_{xx},u_x^2]\neq 0 \).

\item \( \rank[f_{xx}]=3 \) if and only if \( \J2[f_{xx},f_{xx},f_{xx}]\neq 0 \).
\end{enumerate}
\end{lemma}
 
\begin{proof}
This is just Lemma~6.1 of \cite{brookfield2021} expressed in terms of  rank.
\end{proof}

 If \( \rank[f_{xx}]\leq 2 \), then \( f_{xx} \) is reducible because
 \[ f_{xx}=a_x^2+b_x^2=(a_x+i b_x)(a_x-i b_x).\] The converse is true since, if \( f_{xx}=a_xb_x \), then \( \J2[f_{xx},f_{xx},f_{xx}]=0 \)\ \cite[Lemma~6.2]{brookfield2021}, and, in fact, \( f_{xx} \) can be written as a sum of two cubes in many ways:
 \[ f_{xx}=\frac1{4a_0b_0}\left( (a_0 a_x+b_0 b_x)^2 - (a_0 a_x-b_0 b_x)^2\right)\] 
 where \( a_0b_0\neq 0 \). Therefore we have the following lemma:
 
 \begin{lemma}\label{lemx55}
 For a quadratic form \( f_{xx} \), the following are equivalent:
 \begin{enumerate}
 \item \( \J2[f_{xx},f_{xx},f_{xx}]=0 \) 
 
 \item \( \rank[f_{xx}]\leq 2 \)
 
 \item \(  f_{xx} \) is reducible
 \end{enumerate}
 \end{lemma} 

\new We also note some properties of the joint concomitants of  \( f_{xx} \) and a linear form \( a_x \). 

\begin{lemma}\label{lem21}
Let  \( f_{xx} \)  and  \( a_x \) be nonzero forms. 
\begin{enumerate}
\item  \( \J[f_{xx},a_x,u_x]= 0 \) if and only if \( a_x^2 \) divides \( f_{xx}\).

\item \( \J2[f_{xx},a_x^2,u_x^2]=0 \) if and only if \( a_x \) divides \( f_{xx}\).

\end{enumerate}
\end{lemma}

\begin{proof}
See Lemmas~5.2(B1) and 5.2(B2) of \cite{brookfield2021}.
\end{proof}

\begin{lemma}\label{lemx66}
Suppose that \( f_{xx} \) is irreducible and \( a_x \) is nonzero.
\begin{enumerate}
\item \(  \J2[f_{xx},f_{xx},a_x^2]=0  \) if and only if \( f_{xx}=a_x c_x+  b_x^2 \) for some forms \( b_x \) and \( c_x \) such that \( \{a_x,b_x,c_x\} \) is linearly independent.

\item \(  \J2[f_{xx},f_{xx},a_x^2]\neq 0  \) if and only if \( f_{xx}=a_0 a_x ^2+ b_x c_x \) for some  forms \( b_x \) and \( c_x \) such that \( \{a_x,b_x,c_x\} \) is linearly independent,  and some nonzero \( a_0\in \C \).
\end{enumerate}
\end{lemma}

\begin{proof}
\begin{enumerate}
\item If \(  \J2[f_{xx},f_{xx},a_x^2]=0  \), then, from  \cite[Lemma~6.4]{brookfield2021},  \( f_{xx}=a_x c_x+ b_x^2 \) for some  forms \( b_x \) and \( c_x \). In this circumstance,  \( \J2[f_{xx},f_{xx},f_{xx}]=-12 [abc]^2  \),  so \( f_{xx} \)  is irreducible if and only if \( [abc]\neq 0 \), if and only if \( \{a_x,b_x,c_x\} \) is linearly dependent. 

\item Let \( g_{xx}=f_{xx} - a_0 a_x^2 \) with \( a_0\in \C \) to be determined. From the identity
\[ \J2[g_{xx},g_{xx},g_{xx}] = 
 \J2[f_{xx}, f_{xx}, f_{xx}] - 3 a_0 \J2[f_{xx}, f_{xx}, a_x^2], \] 
 we see that \( a_0\in \C \) can be chosen so that \( \J2[g_{xx},g_{xx},g_{xx}]  = 0 \). By \lref{lemx55}, \( g_{xx}=b_xc_x \) for some  forms \( b_x \) and \( c_x \) and so \( f_{xx}=a_0 a_x ^2+ b_x c_x \).  In this circumstance,  \( \J2[f_{xx},f_{xx},f_{xx}]=-12 a_0 [abc]^2  \),  so \( f_{xx} \)  is irreducible if and only if \( a_0 [abc]\neq 0 \), if and only if \( \{a_x,b_x,c_x\} \) is linearly dependent and \( a_0\neq 0 \). 
 \qedhere

\end{enumerate}
\end{proof}

The two conditions in this lemma have an important geometric interpretation. When \( \{a_x,b_x,c_x\} \) is independent, the lines \( a_x=0 \), \( b_x=0 \) and \( c_x=0 \) form a triangle with three distinct vertices.  If  \( f_{xx}=a_x c_x+ b_x^2 \), then  the line \( a_x=0 \) and the conic \( f_{xx}=0 \) intersect at a single point, namely,  the intersection of \( a_x \) and \( b_x \). Hence \( a_x=0 \) is tangent to \( f_{xx}=0 \) at that point.

In the other case, \( f_{xx}=a_0 a_x ^2+ b_x c_x \),   the line \( a_x=0 \) and the conic \( f_{xx}=0 \) intersect at two distinct points, namely,  the intersection of \( a_x \) and \( b_x \), and the intersection of  \( a_x \) and \( c_x \). Hence \( a_x=0 \) is a secant line of \( f_{xx}=0 \).

Thus we have the following geometric lemma. 
\begin{lemma}\label{lemx77}
Suppose that \( f_{xx} \) is irreducible and \( a_x \) is nonzero.
\begin{enumerate}
\item \(  \J2[f_{xx},f_{xx},a_x^2]=0  \) if and only if \( a_x=0 \) is a tangent line of  \( f_{xx}=0 \).

\item \(  \J2[f_{xx},f_{xx},a_x^2]\neq 0  \) if and only if \( a_x=0 \) is a secant line of  \( f_{xx}=0 \).
\end{enumerate}
\end{lemma}

\new Note the difference between  \( \J2[f_{xx},f_{xx},u_x^2]=0  \) and  \( \J2[f_{xx},f_{xx},a_x^2]=0  \) in Lemmas~\ref{lem10}\upitem1 and~\ref{lemx66}\upitem1. In the first case, the coefficients of \( u_x \) are variables, so  \( \J2[f_{xx},f_{xx},u_x^2] \) is a quadratic form in \( \uu \) with six coefficients. The assumption that \( \J2[f_{xx},f_{xx},u_x^2]=0  \) means that all six coefficients are zero. In contrast, the coefficients of \( a_x \) are fixed, so \( \J2[f_{xx},f_{xx},a_x^2]  \) is a single complex number.   Hence \( \J2[f_{xx},f_{xx},a_x^2] =0 \)  is much weaker condition. 

\section{The Concomitants of $f_{xxx}$}
 
The most important concomitants of a ternary cubic form \( f_{xxx} \) are defined by
 \begin{equation}\label{eq837}
\begin{aligned}
\theta_{uuxx}&=\frac14\J2[f_{xxx},f_{xxx},u_x^2]&
\Delta_{xxx}&=\frac 1{12}\J2[f_{xxx},f_{xxx},f_{xxx}]\\
S_{uuu}&=\frac1{576}\J4[f_{xxx} u_x,f_{xxx}u_x,f_{xxx}u_x]&
S&=\Vert S_{uuu} \Vert_{u^3\mapsto f}\\
T_{uuu}&=\frac1{576}\J4[f_{xxx} u_x,f_{xxx}u_x,\Delta_{xxx}u_x]&
T&=\Vert T_{uuu} \Vert_{u^3\mapsto f}\\
\end{aligned}
\end{equation}
\[ F_{6u} =\frac1{3072} \J2[\theta_{uuxx},\theta_{uuxx},u_x^2] \]
and are discussed in much  detail in \cite[Section~7]{brookfield2021}.  Following the convention for subscripts, \( \Delta_{xxx} \)  is a cubic form in \( \xx \) (called the \textbf{Hessian} of \( f_{xxx} \)), \( \theta_{uuxx} \) is quadratic form in both \( \xx \) and \( \uu \), \( S_{uuu} \) and \( T_{uuu} \) are cubic forms in \( \uu \), and \( F_{6u}=F_{uuuuuu} \) has degree \( 6 \) in \( \uu \).  \( S \) and \( T \) are the Aronhold invariants of \( f_{xxx} \) and are constants (i.e.~degree zero in \( \xx \) and \( \uu \)). Here \( S \) is obtained from \( S_{uuu} \) by replacing the coefficients of \( u_x^3 \) by the corresponding coefficients of \( f_{xxx} \), a substitution abbreviated as \( u^3\mapsto f \).

\new The role these concomitants play in determining  rank can already be seen in their values on sums of cubes: 

\begin{lemma}\label{lem227}
Let \(a_x\), \(b_x\), \(c_x\) and \(d_x\) be linear forms.

\begin{enumerate}

\item \label{iione} If \(f_{xxx}=a_x^3\), then 
\[ \theta_{uuxx}=F_{6u}=\Delta_{xxx}=S=T=0\] 

\item \label{iitwo} If \(f_{xxx}=a_x^3+b_x^3 \), then
\[ \theta_{uuxx}=36   [abu]^2 a_xb_x\quad S=T=0 \qquad 
F_{6u}=-27 [abu]^6 \qquad
\Delta_{xxx}=0  \] 

\item \label{iithree}  If \(f_{xxx}=a_x^3+b_x^3+c_x^3 \), then
\begin{align*}
\Delta_{xxx}&= 108  [abc]^2  a_x b_x c_x & S&=0 \\ 
2S_{uuu}&= -27 [abc][abu][bcu][cau] & T&= -5832  [abc]^6 
\end{align*}

\item  \label{iifour} If  \( f_{xxx}=a_x^3+b_x^3+c_x^3+d_x^3 \), then
\( S= 1296 \, [a b c] \, [a b d] \, [a c d]  \,[b c  d] \).
\end{enumerate}
\end{lemma}

So, for example, if \( S\neq 0 \), then \( f_{xxx} \) cannot be a sum of three cubes as in \lref{lem227}\eqref{iithree}, so the  rank of \( f_{xxx} \) must be at least four. This, and a few similarly easy consequences of \lref{lem227}, are collected in the following lemma.

\begin{lemma} \label{lem2145}
Let \(f_{xxx}\) be a cubic form.
\begin{enumerate}

\item If \(  \theta_{uuxx}\neq 0 \), then \(\rank[f_{xxx}]\geq 2\).

\item If \(\Delta_{xxx}\neq 0\), then \(\rank[f_{xxx}]\geq 3\).

\item If \( F_{6u}=0   \) and \(  \theta_{uuxx}\neq 0 \), then \(\rank[f_{xxx}]\geq 3\).

\item If \(S\neq 0\), then \(\rank[f_{xxx}]\geq 4\).

\item If \(T=0\) and \(\Delta_{xxx}\neq 0\), then \(\rank[f_{xxx}]\geq 4\).
\end{enumerate}
\end{lemma}

\new The following lemma collects other properties of cubic forms from \cite{brookfield2021} that are needed for this article. 

\begin{lemma}\label{lem22}
Let \( f_{xxx} \) be a cubic form and let  \( a_x \), \( b_x \) and \( c_x \) be linear forms.
\begin{enumerate}\itemsep=2pt

\item \label{cub1} \( \theta_{uuxx}=0 \) if and only if \( f_{xxx}=a_x^3 \) for some form \( a_x \).

\item \label{cub8} \( F_{6u}=0 \) if and only if \( f_{xxx}=a_x^2b_x \) for some  forms \( a_x \) and \( b_x \).

\item \label{cuc1} \( \Delta_{xxx}=0 \), if and only if \( S_{uuu}=0 \), if and only if   \textup(\( f_{xxx}=a_x^3+b_x^3 \)  or \( f_{xxx}=a_x^2 b_x \) for some  forms \( a_x \) and \( b_x \)\textup), if and only if   \( f_{xxx}=a_xb_xc_x  \)  for some  linearly dependent  forms \( \{a_x,b_x,c_x\} \).

\item \label{cub3}  \( f_{xxx} \) is completely reducible if and only if \( \{f_{xxx},\Delta_{xxx}\} \) is linearly dependent.

\item  \label{iiabc} If \(f_{xxx}=a_xb_xc_x\), then 
\(  F_{6u}=[abu]^2[bcu]^2[cau]^2 \), 
 \( \Delta_{xxx}=[abc]^2\, a_xb_xc_x  \),  
\( S=[abc] ^4 \) and
\( T=[abc] ^6  \).

\item \label{cub2} If \( a_x \) is nonzero, then  \( \J3[f_{xxx},a_x^3,u_x^3]=0 \) if and only if \( a_x \) divides \( f_{xxx}\).

\item  \label{cub4} If \( \{a_x,b_x\} \) is linearly independent, then \( \J3[f_{xxx},a_x^2b_x,u_x^3]=0 \) if and only if \( f_{xxx}=  a_x^2 c^{\phantom{2}}_x + b^{\phantom{2}}_0 b_x^3\) for some  form \( c_x \) and \( b_0\in \C \).

\item  \label{cub6} If \( \{a_x,b_x,c_x\} \) is linearly independent, then  \( \J3[f_{xxx},a_xb_xc_x,u_x^3]=0 \) if and only if  \( f_{xxx}=a_0 a_x^3+ b_0 b_x^3+c_0 c_x^3+d_0 a_x b_x c_x \) for some \( a_0,b_0,c_0,d_0\in \C \).

\item \label{cub5} \( \J3[f_{xxx},\Delta_{xxx},u_x^3]=0 \)

\item \label{cub7} If \( \{a_x,b_x,c_x\} \) is linearly independent, then \[ \B_3=\{a_x^3,b_x^3,c_x^3,a_x^2 b_x,a_x^2 c_x , b_x^2 a_x,b_x^2 c_x, c_x^2 a_x, c_x^2 b_x,a_x b_x c_x\} \] is a basis for the vector space of cubic forms.
\end{enumerate}
\end{lemma}

\begin{proof}
See Theorem~9.2, Lemma~5.2(C3), Lemma~5.5, Lemma~5.4 and Sections~5 and~7 of~\cite{brookfield2021}.
\end{proof}

From \upitem1--\upitem3 of this lemma we have the implications
\begin{equation}\label{eq33}
f_{xxx} =0\Rightarrow \theta_{uuxx}=0 \Rightarrow F_{6u}=0 \Rightarrow (\Delta_{xxx}=0 \Leftrightarrow S_{uuu}=0)
\end{equation}

\new For a cubic form \( f_{xxx} \), its Hessian \( \Delta_{xxx} \) is also a cubic form, so  has its own \( \theta_{uuxx} \), \( \Delta_{xxx} \) and \( F_{6u} \) concomitants. We will denote these by \( \btheta_{uuxx} \), \( \bDelta_{xxx} \) and \( \bF_{6u} \), that is, 
\[  \btheta_{uuxx}=\Vert\theta_{uuxx}\Vert_{f\mapsto \Delta} \qquad  \bDelta_{xxx}  =\Vert \Delta_{xxx}\Vert_{f\mapsto \Delta} \qquad \bF_{6u} =\Vert F_{6u}\Vert_{f\mapsto \Delta} \] 
These concomitants of \( \Delta_{xxx} \) are related to the concomitants of \( f_{xxx} \) by
\begin{align}
9 \, \btheta_{uuxx}&=
  8 \con_{ux}[T_{uuu} f_{xxx}]- 2 T u_x^2 -  S \theta_{uuxx } \label{eq2143}\\
 \bDelta_{xxx} & =3S^2 f_{xxx}-2 T \Delta_{xxx}\label{eq2144}\\
\bF_{6u}  &=12 S\, T_{uuu} S_{uuu} - 3 S^2 F_{6u}  - 8 T S_{uuu}^2 \label{eq2145}
\end{align}
Though we do not  need this fact, we mention that the  \( S_{uuu} \), \( T_{uuu} \), \( S \) and \( T \) concomitants of \( \Delta_{xxx} \) are related to the concomitants of \( f_{xxx} \) by
\begin{align*}
 \bS&= 4 T^2 - 3 S^3 &
 \bS_{uuu}  &=4 T S_{uuu}- 3 S T_{uuu}  \\
 \bT&= T (9 S^3 - 8 T^2) &
 \bT_{uuu}&= 6  S T  T_{uuu} - (8 T^2 - 3 S^3 ) S_{uuu}
\end{align*}
The implications in \eqref{eq33} when applied to \( \Delta_{xxx} \) become
\begin{equation}\label{eq38}
\Delta_{xxx} =0\Rightarrow \btheta_{uuxx}=0 \Rightarrow \bF_{6u}=0 \Rightarrow (\bDelta_{xxx}=0 \Leftrightarrow \bS_{uuu}=0)
\end{equation}

\new In the next two lemmas we prove other important relationships between the  concomitants of \( f_{xxx} \) and those of \( \Delta_{xxx} \).
\begin{lemma}\label{lem2139}
For a cubic form \(f_{xxx}\) the following are equivalent:
\begin{enumerate}
\item \(\bF_{6u} =0\)

\item \( \bDelta_{xxx}  =0\)

\item \(S=T=0\)
\end{enumerate}
\end{lemma}

\begin{proof}

(1)\( \Rightarrow \)(2): This follows from \eqref{eq38}.

\noindent(2)\( \Rightarrow \)(3):  There are two cases: If \( f_{xxx}\) is not completely reducible, then \(\{f_{xxx},\Delta_{xxx}\}\) is linearly independent by \lref{lem22}\eqref{cub3}. So, if  \( \bDelta_{xxx}  =0\), then  \(S=T=0\) follows from \eqref{eq2144}. 

Otherwise, suppose that  \(f_{xxx}\) is completely reducible, that is, \(f_{xxx}=a_xb_xc_x\) for some linear forms \( a_x \), \( b_x \) and \( c_x \).  Then, by \lref{lem22}\eqref{iiabc},  \(\Delta_{xxx}=[abc]^2 a_xb_xc_x\) and \(  \bDelta_{xxx}  = [abc]^4a_xb_xc_x\). If \(\bDelta_{xxx}  =0\), then \([abc]=0\) and so \(S=T=0\) follows from \lref{lem22}\eqref{iiabc}.

\noindent(3)\( \Rightarrow \)(1): This follows directly from \eqref{eq2145}.

\end{proof}

\begin{lemma}\label{lem2555}
For a cubic form \(f_{xxx}\) the following are equivalent:
\begin{enumerate}
\item \(T_{uuu} =0\)

\item \( \btheta_{uuxx}  =0\)

\end{enumerate}
\end{lemma}

\begin{proof}
(1)\( \Rightarrow \)(2):  If  \(T_{uuu} =0\), then the identities
 \[ 6 \,T=\con_{ux}^3[T_{uuu} f_{xxx}]  \qquad 6 S^2= \con_{ux}^3[T_{uuu} \Delta_{xxx}] \]
imply  \( S=T=0 \) and then \eqref{eq2143} gives \( \btheta_{uuxx} =0\).

(2)\( \Rightarrow \)(1): If \( \btheta_{uuxx}  =0 \), then  \( \bDelta_{xxx}  =0\) follows from \eqref{eq38},  \( S=T=0 \) follows from \lref{lem2139}, and \( \con_{ux}[T_{uuu} f_{xxx}]=0 \) follows  from \eqref{eq2143}. From the identity 
\[  \J2[\theta_{uuxx}, \con_{ux}[f_{xxx} T_{uuu}], u_x^2] = 48 S_{uuu} T_{uuu} \] 
we get \( S_{uuu}T_{uuu}=0 \), so either  \( S_{uuu}=0 \) or \( T_{uuu}=0 \).  Because of the identity \[ 12 T_{uuu} = \con_{ux}^2[S_{uuu} \theta_{uuxx}] ,\] we have \( T_{uuu}=0 \) in both cases. 
\end{proof}

Combining the previous two lemmas with \eqref{eq33}  and \eqref{eq38},  we get the following implications:
\begin{multline}\label{eq34}
f_{xxx} =0\Rightarrow \theta_{uuxx}=0 \Rightarrow F_{6u}=0 \Rightarrow (\Delta_{xxx}=0 \Leftrightarrow S_{uuu}=0)\\ \Rightarrow (\btheta_{uuxx}=0 \Leftrightarrow T_{uuu}=0)\Rightarrow (\bF_{6u}=0 \Leftrightarrow \bDelta_{xxx}=0 \Leftrightarrow S=T=0)
\end{multline}
 
 \new There are two  types of cubic forms with geometric significance that are of particular interest for this article. Suppose that \( \{a_x,b_x,c_x\} \) are linearly independent  forms. 
\begin{itemize}
\item If \(f_{xxx}=a_x(a_x c_x+ b_x^2)  \), then the equation \( f_{xxx}=0 \) represents  an irreducible conic \( a_x c_x+ b_x^2 =0\) and with a tangent line \( a_x =0 \). The point of tangency is the intersection of the lines \( a_x=0 \) and \( b_x=0 \). See Lemmas~\ref{lemx66} and~\ref{lemx77}.

\item If \( f_{xxx}=a_x^2 c_x+ b_x^3 \), then the equation \( f_{xxx}=0 \) represents  an irreducible cubic curve with a cusp at the intersection of the lines \( a_x=0 \) and \( b_x=0 \). 
\end{itemize} 

The values of the concomitants for forms of these types are provided by the following lemma.

\begin{lemma}\label{lem228}
Let \(a_x\), \(b_x\) and \(c_x\) be  linear forms.

\begin{enumerate}
\item  \label{iitang} If \(f_{xxx}=a_x(a_x c_x+ b_x^2)  \), then 
\[ \Delta_{xxx}= -4 [abc]^2 a_x^3 \qquad S_{uuu}= - 2[abc] [acu]^3 \quad T_{uuu}=S=T= \btheta_{uuxx}=0\] 

\item  \label{iicusp} If \( f_{xxx}=a_x^2 c_x+ b_x^3 \), then 
\[ \Delta_{xxx}= -12  [abc]^2 a_x^2 b_x \qquad S=T=0 \qquad \btheta_{uuxx}=-576  [abc]^4 [abu]^2 a_x^2\] 
\end{enumerate}
\end{lemma}

\section{The Rank of Cubic Forms}\label{sectmain}

The goal of this section is to provide conditions on the concomitants of a cubic form \( f_{xxx} \) that determine its  rank.  The table below summarizes these results.

 \[ \def\strut{\vrule height 11pt width 0pt depth 6pt}
\begin{array}{|c|c|l|c|}\hline
\strut   \text{Condition} & \text{Rank} & \hfil f_{xxx}  & \text{Lemma}\\\hline
 \strut  f_{xxx}=0 & 0 & 0 &\\\hline
\strut   \theta_{uuxx}=0 &1 & a_x^3 \text{ with }   a_x\neq 0 & \ref{lem31} \\\hline
\strut   F_{6u}=0 & 3 & a_x^2 b_x   \text{ with } \{a_x,b_x\} \text{ independent}& \ref{lem32} \\\hline
\vrule height 16pt width 0pt depth 11pt   \parbox{50pt}{\centering\(\Delta_{xxx}=0\) \( S_{uuu}=0 \)  } & 2 & a_x^3+  b_x^3  \text{ with }  \{a_x,b_x\} \text{ independent}& \ref{lem33}\\\hline
\vrule height 16pt width 0pt depth 11pt   \parbox{50pt}{\centering\(\btheta_{uuxx}=0\)  \( T_{uuu}=0 \)  }  & 5 & a_x(a_x c_x+ b_x^2  )   \text{ with }   \{a_x,b_x,c_x\} \text{ independent}& \ref{lem34} \\\hline
\vrule height 22pt width 0pt depth 16pt   \parbox{50pt}{\centering\(\bF_{6u} =0\)  \( \bDelta_{xxx}=0 \)  \( S=T=0\) }& 4 & a_x^2 c^{\phantom{2}}_x + b_x^3  \text{ with }  \{a_x,b_x,c_x\} \text{ independent}& \ref{lem35}\\\hline
\strut   S =0& 3 &  a_x^3+ b_x^3+c_x^3  \text{ with }  \{a_x,b_x,c_x\} \text{ independent}& \ref{lem36}\\\hline
\vrule height 17pt width 0pt depth 12pt   S\neq 0 & 4 &\parbox{2.5in}{\(  a_x^3+  b_x^3+c_x^3+d_x^3  \) with \(  \{a_x,b_x,c_x,d_x\}  \) \\triplewise independent} & \ref{lem37} \\\hline 
\end{array}
 \] 
 
Where there is more than one condition in an entry in the first column, the conditions are equivalent (Lemmas~\ref{lem22}\eqref{cuc1}, \ref{lem2139} and \ref{lem2555}). 
The condition in each row of the table goes together with the negation of the condition in the row above. For example the fifth row of the table could be read as, \( \Delta_{xxx}= 0 \) and \( F_{6u}\neq 0 \) if and only if \(f_{xxx}= a_x^3+  b_x^3 \) for some linearly independent  forms \( \{a_x,b_x\} \), and in this circumstance, \( \rank[f_{xxx}]=2 \) (\lref{lem33}).   The implications in \eqref{eq34} ensure that we are not missing any special cases---all cubic forms are described by entries in this table.

\begin{lemma}\label{lem31}
For a cubic form \(f_{xxx}\), the following are equivalent:
\begin{enumerate}
\item \(\theta_{uuxx}=0\) and  \(f_{xxx} \neq 0\).

\item \(f_{xxx}=a_x^3\) for some nonzero  form \(a_x\).

\item  \(\rank[f_{xxx}]=1 \)
\end{enumerate}
\end{lemma}

\begin{proof}
This is immediate from \lref{lem22}\eqref{cub1}.
\end{proof}

\begin{lemma}\label{lem33}
For a cubic form \(f_{xxx}\), the following are equivalent:
\begin{enumerate}
\item  \(\Delta_{xxx}=0\) and \(F_{6u} \neq 0\).

\item \(f_{xxx}=a_x^3+ b_x^3\) for some linearly independent  forms  \(\{a_x,b_x\}\).

\item  \(f_{xxx}=a_xb_xc_x  \) for some linearly dependent, but pairwise independent forms \(\{a_x,b_x,c_x\}\).

\item  \(\rank[f_{xxx}]=2 \).
\end{enumerate}
\end{lemma}

\begin{proof}
(1)\( \Rightarrow \)(2): From \lref{lem22}\eqref{cub8} and \eqref{cuc1} we have \(f_{xxx}=a_x^3+ b_x^3\) for some   forms  \( a_x \) and \( b_x \). From \lref{lem227}\eqref{iitwo}, \( F_{6u}\neq0 \) implies that \( [abu]\neq 0 \) and so \(\{a_x,b_x\}\) is linearly independent.

(2)\( \Rightarrow \)(1): This follows directly from \lref{lem227}\eqref{iitwo}. In particular, since \(\{a_x,b_x\}\) is linearly independent, we have \( [abu]\neq 0 \) and so  \(F_{6u} \neq 0\). 

(1)\( \Rightarrow \)(3):   From \lref{lem22}\eqref{cuc1} we have  \( f_{xxx}=a_xb_xc_x  \)  for some  linearly dependent  forms \( \{a_x,b_x,c_x\} \). Because \( F_{6u} \neq 0\), \lref{lem22}\eqref{iiabc} implies that \( [abu] \), \( [bcu] \) and \( [cau] \) are all nonzero. Hence \(\{a_x,b_x,c_x\}\) is pairwise independent.

(3)\( \Rightarrow \)(1): This follows  from \lref{lem22}\eqref{iiabc}. Since  \( \{a_x,b_x,c_x\} \) is dependent, we have \( [abc]=0 \) and hence \( \Delta_{xxx}=0 \). Also, since \(\{a_x,b_x,c_x\}\) is pairwise independent, \( [abu] \), \( [bcu] \) and \( [cau] \) are all nonzero and so  \(F_{6u} \neq 0\). 

(1,2)\( \Rightarrow \)(4):  The rank of \( f_{xxx} \) cannot be one because \( F_{6u}= 0 \) for such forms by \lref{lem227}\eqref{iione}. Since \( f_{xxx} \) is a sum of two cubes, we have \( \rank[f_{xxx}]=2 \).

(4)\( \Rightarrow \)(2): Since \( \rank[f_{xxx}]=2 \),  \(f_{xxx}=a_x^3+ b_x^3\) for some  forms  \(\{a_x,b_x\}\). These  forms must be independent, since otherwise the two terms could be combined and \( f_{xxx} \) would have rank one.
\end{proof}

Applying  \lref{lem33}  to the cubic form \( \Delta_{xxx} \), we see that \( \Delta_{xxx} \) can never have  rank two.  If that were possible, from \lref{lem33}, we would have \( \bDelta_{xxx}=0 \) and \( \bF_{6u}\neq 0 \), contradicting \lref{lem2139}. In other words,  no rank two cubic form is  the Hessian of another cubic form. This is contrary to the claim made by many 19th century mathematicians that each cubic form is the Hessian of three other cubic forms \cite[p.~153]{salmon1879}.

\new If \( f_{xxx} \)  has rank 2, then  \(f_{xxx}=a_x^3+ b_x^3\) for some linearly independent  forms  \(\{a_x,b_x\}\), and, from \lref{lem227}\eqref{iitwo}, \( \theta_{uuxx}=36   [abu]^2 a_xb_x \).  The uniqueness of factorizations means that \( a_x \) and \( b_x \) are uniquely determined by \( f_{xxx} \) up to order and scalar multiplication. This also shows how to express any rank two cubic form as a sum of two cubes.

\begin{method}\label{meth23}
Suppose that \( f_{xxx} \) is a cubic form with rank two. Fix \( (\uu)\in \C^3 \) so that  \( \theta_{uuxx}\) is  nonzero. Write \( \theta_{uuxx}=a_xb_x \) for some linearly independent  forms \( \{a_x,b_x\} \).  Then  \(f_{xxx}=a_0a_x^3+b_0b_x^3 \) for some \( a_0,b_0\in \C \) that can be determined by matching coefficients.
\end{method}

\begin{example}\label{ex36}
Suppose that \( f_{xxx}= x_1^2 x_2-x_2^3\). Then \( \Delta_{xxx}=0 \) and \( F_{6u}=4 u_3^6  \), so, by \lref{lem33}, \( \rank[f_{xxx}]=2 \). To write \( f_{xxx} \) as a sum of two cubes, we factor \( \theta_{uuxx}\):
\[ 
 \theta_{uuxx}= -4 u_3^2 ( x_1^2 +3 x_2^2)=  -4 u_3^2 \left( x_1 - i \sqrt{3}\,x_2\right)\left( x_1 + i \sqrt{3}\,x_2\right)
 \] Hence \( f_{xxx} \) is a linear combination of \( \left( x_1 - i \sqrt{3}\,x_2\right)^3 \) and \( \left( x_1 + i \sqrt{3}\,x_2\right) ^3 \). By matching coefficients, we find
\[  f_{xxx} = \frac{1}{6i\sqrt{3} }
   \left( \left( x_1 - i \sqrt{3}\,x_2\right) ^3- \left( x_1 + i \sqrt{3}\,x_2\right) ^3\right). \]
Of course,  \lref{lem33} also says that \( f_{xxx} \) can be written as a product of three linearly dependent  forms. For this simple example, this factorization is easy to find: \(  f_{xxx}=x_2(x_1+x_2)(x_1- x_2) \).
\end{example}

\begin{example}\label{ex49}
As in \lref{lem33}\upitem3, suppose \(f_{xxx}=a_xb_xc_x  \) for some linearly dependent, but pairwise independent forms \(\{a_x,b_x,c_x\}\). We could use \mref{meth23} to write \( f_{xxx} \) as a sum of two cubes, but factoring \( \theta_{xxuu}  \) in the presence of the linear dependence condition is tricky. Instead we can proceed as follows.

 If it happened that \( a_x+b_x+c_x=0 \), then the identity
\begin{equation}\label{eq130}
9 a_x b_x c_x = (1 + 2 \omega )\left((a_x - \omega ^2  b_x)^3 - (a_x - \omega  b_x)^3\right) + 9  a_x  b_x (a_x + b_x + c_x)
\end{equation}
with \( \omega=e^{2\pi i/3} \), could be used to express \( f_{xxx}=a_xb_xc_x \) as a sum of two cubes. 

In the general case,  because \(\{a_x,b_x,c_x\}\) is dependent, but pairwise independent, we  have \( a_0 a_x+b_0 b_x+c_0 c_x=0 \) for some \( a_0,b_0,c_0\in \C \), none of which are zero. If we replace \( a_x \) by \( a_0a_x \),  \( b_x \) by \( b_0b_x \), and \( c_x \) by \( c_0c_x \) in \eqref{eq130},  and set \( a_0 a_x+b_0 b_x+c_0 c_x=0 \), we get 
\[ 9 a_0 b_0 c_0 a_x b_x c_x = (1 + 2 \omega )\left((a_0 a_x - \omega ^2  b_0 b_x)^3 - (a_0 a_x - \omega b_0 b_x)^3\right) 
 \] which can be solved for \( f_{xxx}=a_xb_xc_x \) to express \( f_{xxx}\) as a sum of two cubes.
\end{example}

\begin{lemma}\label{lem32}
For a cubic form \(f_{xxx}\), the following are equivalent:
\begin{enumerate}
\item  \(F_{6u}=0\) and \(\theta_{uuxx} \neq 0\).

\item \(f_{xxx}=a_x^2 b_x\) for some linearly independent  forms  \(\{a_x,b_x\}\).
\end{enumerate}
When these conditions hold,  \(\rank[f_{xxx}]=3 \).
\end{lemma}

\begin{proof}
(1)\( \Leftrightarrow \)(2): From \lref{lem22}\eqref{cub8}, we have  \( F_{6u}=0 \) if and only if \( f_{xxx}= a_x^2 b_x \) for some  forms  \(a_x \) and \(b_x\). But if \( f_{xxx}= a_x^2 b_x \), then \( \theta_{uuxx}=-4 [abu]^2a_x^2 \). So \( \theta_{uuxx}\neq 0 \) if and only if \( [abu]\neq0 \), if and only if  \(\{a_x,b_x\}\) is linearly independent. 

For the final claim, the identity
\begin{equation}\label{eq2118}
6 a_x^2 b_x =(a_x + b_x)^3 - (a_x - b_x)^3 - 2 b_x^3
\end{equation}
implies that \( \rank[f_{xxx}] \leq 3 \), and  \lref{lem2145}\upitem3 implies that \( \rank[f_{xxx}]\geq 3 \).
\end{proof}

As in the above lemma, suppose that \( f_{xxx}=a_x^2 b_x \)  with   \( \{a_x,b_x\} \) linearly independent and \( \rank[f_{xxx}]=3 \). It turns out that there are infinitely many ways of expressing \(  f_{xxx}  \) as a sum of three cubes. To show this, we consider how to find   a form \( c_x \) so that \( g_{xxx}=f_{xxx}- c_x^3 \) has rank \( 2 \).  Since \( \rank[f_{xxx}]=3 \),  \( c_x \) must be nonzero. By \lref{lem33}, \( g_{xxx} \) has rank \( 2 \) if and only if \( \Vert \Delta_{xxx}\Vert_{f\mapsto g}=0 \) and \( \Vert F_{6u }\Vert _{f\mapsto g} \neq 0\).  

By direct calculation, \(  \Vert \Delta_{xxx}\Vert_{f\mapsto g} = 12 [abc]^2a_x^2c_x \). So  \( \Vert \Delta_{xxx}\Vert_{f\mapsto g}=0 \) if and only if  \( [abc]= 0 \) if and only if  \( c_x \) is a linear  combination of \( \{a_x,b_x\} \), that is \( c_x=a_0a_x+b_0b_x \) for some \( a_0,b_0\in \C \). With this extra assumption, 
\[ \Vert F_{6u }\Vert _{f\mapsto g} = b_0^3 (4 - 27 a_0^2 b_0) [abu]^6.\] 
Thus, \( g_{xxx} \) has rank \( 2 \) if and only if \(b_0 (4 - 27 a_0^2 b_0)\neq0 \). Indeed, in this circumstance, using \mref{meth23}, we get
\begin{equation}\label{eqx44}
g_{xxx} =  \frac{(3 a_0 b_0 +  \gamma )}{16 b_0^3  \gamma } ( (a_0 b_0 - \gamma ) a_x - 2 b_0^2 b_x)^3 - \frac{(3 a_0 b_0 -  \gamma )}{16 b_0^3  \gamma }  ( (a_0 b_0 + \gamma ) a_x - 2 b_0^2 b_x)^3 
\end{equation}
 where \( \gamma \in \C \) is nonzero and  satisfies \(  b_0 (4 - 27 a_0^2 b_0 ) + 3  \gamma ^2=0 \).
Since \( f_{xxx}=g_{xxx}+c_x^3 \), this expression for \( g_{xxx} \) can be used to write \( f_{xxx} \) as a sum of three cubes in infinitely many ways.
 
 \begin{method}\label{meth2}
Suppose that \( f_{xxx}=a_x^2 b_x \) for some linearly independent  forms \( \{a_x,b_x\} \). Fix \( a_0,b_0\in \C \)  such that \(b_0 (4 - 27 a_0^2 b_0 )\neq0 \). Set \( c_x=a_0 a_x+b_0b_x \)  and \( g_{xxx}=f_{xxx}- c_x^3 \). Then \( \rank[g_{xxx}]=2 \) so \( g_{xxx} \)  can be written as a sum of two cubes using \mref{meth23} or \eqref{eqx44}. Then \( f_{xxx}=g_{xxx}+c_x^3 \) is expressed as a sum of three cubes.
\end{method}

\begin{example}
Consider \( f_{xxx}=x_1^2 x_2 \). From \lref{lem32}, \( \rank[f_{xxx}]=3 \) and \eqref{eq2118} becomes
\[ 6f_{xxx}= (x_1+x_2)^3-(x_1-x_2)^3-2 x_2^3.\] Using \mref{meth2}, there are other ways of expressing \( f_{xxx} \) as a sum of  three cubes.   For example, suppose that  \( c_x=x_2 \). Then 
\( g_{xxx}=x_1^2 x_2-x_2^3 \)
with \( \Vert F_{6u }\Vert _{f\mapsto g} =4u_3^6\neq 0\) so \( \rank[g_{xxx}]=2 \). We can use \eqref{eqx44} or \eref{ex36} to express \( g_{xxx} \) as a sum of two cubes  
and hence \(f_{xxx}\) as a sum of three cubes:
\[ f_{xxx} =g_{xxx}+c_x^3=  \frac{1}{6i\sqrt{3} }
   \left( \left( x_1 - i \sqrt{3}\,x_2\right) ^3- \left( x_1 + i \sqrt{3}\,x_2\right) ^3\right)+ x_2^3\] 
\end{example}

\begin{lemma}\label{lem35}
\textup[Cusp form\textup] For a cubic form \(f_{xxx}\), the following are equivalent:
\begin{enumerate}
\item \(S=T=0\) and  \(\btheta_{uuxx} \neq 0\).

\item  \(\bF_{6u}=0\) and \(\btheta_{uuxx} \neq 0\).

\item \(\Delta_{xxx}=a_x^2 b_x \) for some  linearly independent  forms  \(\{a_x,b_x\}\).

\item  \(f_{xxx}=  a_x^2 c^{\phantom{2}}_x +  b_x^3\) for some linearly independent  forms  \(\{a_x,b_x,c_x\}\).

 \item \( S=0 \) and  \( \rank[ f_{xxx}]=4 \).
\end{enumerate}
\end{lemma}

\begin{proof}
(1) and (2) are equivalent because of \lref{lem2139}. The equivalence of (2) and (3) is just \lref{lem32} with \( f_{xxx} \) replaced by \( \Delta_{xxx} \). 

(3)\( \Rightarrow \)(4): If (3) holds, then  \( \J3[f_{xxx},a_x^2b_x,u_x^3]=0 \) by  \lref{lem22}\eqref{cub5}, and \(f_{xxx}=  a_x^2 c^{\phantom{2}}_x +  b_x^3\) follows from \lref{lem22}\eqref{cub4} after scaling \( b_x \).  The forms  \(\{a_x,b_x,c_x\}\) must be  linearly independent, since otherwise \( [abc]=0 \) and  \(  \Delta_{xxx}  = 0\) by \lref{lem228}\eqref{iicusp}.

(4)\( \Rightarrow \)(1): If \upitem4 holds, then \upitem1 follows from \lref{lem228}\eqref{iicusp} since \( [abc]\neq 0 \).

So far we have proved the equivalence of \upitem1, \upitem2, \upitem3 and \upitem4. It remains only to include \upitem5 in these equivalent conditions.

(1,3)\( \Rightarrow \)(5): From (1) and (3), we have \(S=T=0\) and  \( \Delta_{xxx}\neq 0\), so from \lref{lem2145}\upitem5,  \(\rank[f_{xxx}]\geq 4\). On the other hand, from \eqref{eq2118}, \(a_x^2 c^{\phantom{2}}_x\) is a sum of three cubes, so  \(f_{xxx}=  a_x^2 c^{\phantom{2}}_x +  b_x^3\) is a sum of four cubes and  \(\rank[ f_{xxx}]=4 \) and  \upitem5 holds.

(5)\( \Rightarrow \)(4): Since \( \rank[f_{xxx}]=4 \), we have  \(  f_{xxx}=a_x^3+b_x^3+c_x^3+d_x^3 \) for some  forms \( a_x \), \( b_x \), \( c_x \) and \( d_x \). From \lref{lem227}\eqref{iifour}, we have  \(S=1296 [a b c] \, [a b d] \, [a c d]  \,[b c  d]\), which, because  \(S=0\), implies that three of the linear forms are linearly dependent. Suppose, without loss of generality, that \([acd]=0\) and write \( g_{xxx}=a_x^3+c_x^3+d_x^3 \). Then, by \lref{lem227}\eqref{iithree}, the Hessian of 
\(g_{xxx}\) is zero. At the same time the  rank of \(g_{xxx}\)  must be at least three so that the  rank of \(  f_{xxx}\) is four. 
 From Lemmas~\ref{lem31}, \ref{lem33} and~\ref{lem32}, (or from the first few rows of the table), this is possible only if \(g_{xxx}= a_x^2 c_x\) (not the same \(a_x\) and  \(c_x\) as in the rest of the proof). Thus  \( f_{xxx}= g_{xxx}+b_x^3=a_x^2 c_x+b_x^3\).
 
 Finally,  \(\{a_x,b_x,c_x\}\) must be independent since otherwise, from \lref{lem228}\eqref{iicusp},  we would have   \( \Delta_{xxx}=0 \)  and then \( \rank[f_{xxx}]\leq 3 \),  by Lemmas~\ref{lem31}, \ref{lem33} and~\ref{lem32}.
\end{proof}

\begin{lemma}\label{lem34}
\textup[Irreducible conic with tangent line\textup] For a cubic form \(f_{xxx}\), the following are equivalent:
\begin{enumerate}
\item  \(\btheta_{uuxx}=0\) and \(\Delta_{xxx} \neq 0\).

\item \(\Delta_{xxx}=a_x^3\) for some nonzero  form \(a_x\).

\item  \(f_{xxx}=a_x(a_x c_x+ b_x^2  )  \) for some linearly independent  forms  \(\{a_x,b_x,c_x\}\).
\end{enumerate}
When these conditions hold,  \(\rank[f_{xxx}]=5 \).
\end{lemma}

\begin{proof}
The equivalence of (1) and (2) is just \lref{lem31} with \( f_{xxx} \) replaced by \( \Delta_{xxx} \).

(2)\( \Rightarrow \)(3):  Since \(\Delta_{xxx}=a_x^3\), from \lref{lem22}\eqref{cub5}, we have \(\J3[f_{xxx},a_x^3,u_x^3]=0\) and then \lref{lem22}\eqref{cub2} gives  \( f_{xxx}=a_xg_{xx} \) for some  form \( g_{xx} \).  With \( f_{xxx}=a_xg_{xx} \),  a direct calculation gives
\[ \Delta_{xxx}=  \frac1{12}\left(4 \J2[g_{xx}, g_{xx}, g_{xx}] a_x^3 - 3 \J2[g_{xx}, g_{xx}, a_x^2] a_x g_{xx}\right) \]
This seems to give two ways that \( \Delta_{xxx}=a_x ^3 \): Either \( g_{xx}  \) is a scalar multiple of \( a_x^2 \)  or \( \J2[g_{xx}, g_{xx}, a_x^2] =0\).  In the first case, \( f_{xxx} \) is a scalar multiple of \( a_x^3 \) and so we would have \( \Delta_{xxx}=0 \) (\lref{lem227}\eqref{iione}), contrary to the current situation. In the second case, by \lref{lemx66}(1), \( g_{xx}=a_x c_x+ b_x^2 \) for some   forms \( b_x \) and \( c_x \). With this expression for \( g_{xx} \), we have  \(f_{xxx}=a_x(a_x c_x+ b_x^2  )  \). From \lref{lem228}\eqref{iitang}, \( \Delta_{xxx}=-4 [abc]^2 a_x^3 \). Since \(\Delta_{xxx}\) is nonzero, \([abc]\neq 0\) and  \(\{a_x,b_x,c_x\}\) is linearly independent.

(3)\( \Rightarrow \)(2): From \lref{lem228}\eqref{iitang}, if \(f_{xxx}=a_x(a_x c_x+ b_x^2  )  \), then \( \Delta_{xxx}=-4 [abc]^2 a_x^3 \). Because  \(\{a_x,b_x,c_x\}\) is linearly independent, we have \( [abc]\neq 0 \) and so \(\Delta_{xxx}\) is a nonzero cube.

For the final claim, from \lref{lem228}\eqref{iitang} we get \( S=T= \btheta_{uuxx}=0 \). Since we also have  \( \Delta_{xxx}\neq 0 \), \lref{lem2145}\upitem5 implies that  \( \rank[f_{xxx}]\geq 4\). Because \( S\) and \(\btheta_{uuxx}\) are zero, \lref{lem35} implies that \( \rank[f_{xxx}]=4 \) is not possible. Finally, we can write  
\[ 3f_{xxx}=(a_x - c_x)^3 + c_x^3 - a_x^2 (  a_x - 3 b_x - 3 c_x). \]
By \eqref{eq2118}, the third term on the right is a sum of three cubes, and so \(f_{xxx}\) is a sum of five cubes and \( \rank[f_{xxx}]=5 \).
\end{proof}

\begin{lemma}\label{lem36}
For a cubic form \(f_{xxx}\), the following are equivalent:
\begin{enumerate}
\item \( S=0\) and \(T\neq 0\)

\item \(f_{xxx}=a_x^3+b_x^3+c_x^3 \) for some linearly independent  forms  \(\{a_x,b_x,c_x\}\).
\end{enumerate}
When these conditions hold,  \(\rank[f_{xxx}]=3\).
\end{lemma}

\begin{proof}
(1)\( \Rightarrow \)(2): 
Here we apply \lref{lem22}\eqref{cub3} and \eqref{iiabc} to the cubic form \( \Delta_{xxx} \) in place of \( f_{xxx} \). Because \( S=0\), \eqref{eq2144} implies that \(\bDelta_{xxx}\)  is a multiple of \(\Delta_{xxx}\). By \lref{lem22}\eqref{cub3}, \( \Delta_{xxx} \) is completely reducible, that is, \( \Delta_{xxx}=a_xb_xc_x \) for some  forms \(\{a_x,b_x,c_x\}\).  Since \( T\neq0 \),  \lref{lem2139} implies that \(\bDelta_{xxx}\)  is nonzero. Because of this and \lref{lem22}\eqref{iiabc}, we have \( [abc]\neq 0 \) and so \(\{a_x,b_x,c_x\}\) is linearly independent.

From \lref{lem22}(\ref{cub6},\ref{cub5}) we have \[\J3[f_{xxx},a_xb_xc_x,u_x^3]=\J3[f_{xxx},\Delta_{xxx},u_x^3]=0\] and   \( f_{xxx}=a_0 a_x^3+ b_0 b_x^3+c_0 c_x^3+d_0 a_x b_x c_x \) for some \( a_0,b_0,c_0,d_0\in \C \).  
But if \( f_{xxx} \) has this form, then a direct calculation gives
\[ \Delta_{xxx}=[abc]^2 \left((d_0^3 + 108 a_0 b_0 c_0) a_x b_x c_x -   3 d_0^2 (a_0 a_x^3+b_0 b_x^3+ c_0 c_x^3)\right). \]
Because  \( \Delta_{xxx}=a_x b_x c_x \), the linear independence of \(\B_3\) (\lref{lem22}\eqref{cub7}) implies that  \( d_0a_0=d_0b_0=d_0c_0=0 \), so either \( a_0=b_0=c_0=0 \) or  \( d_0=0 \).  In the first case, \(f_{xxx}=d_0 a_x b_x c_x \) is completely reducible.  But for completely reducible cubic forms, \( S=0 \) implies \( T=0 \) (\lref{lem22}\eqref{iiabc}), so this case is eliminated. Hence we have \( d_0=0 \) and  \( f_{xxx}= a_0 a_x^3+b_0 b_x^3+ c_0 c_x^3\).

Finally, because \(T\) is nonzero, \( f_{xxx} \) cannot be a sum of two cubes (\lref{lem227}\eqref{iitwo}), so \(a_0\), \(b_0\), \(c_0\) must be nonzero. By scaling \( a_x \), \( b_x \) and \( c_x \), \( f_{xxx} \) can be written in the claimed form. 

(2)\( \Rightarrow \)(1): This follows directly from \lref{lem227}\eqref{iithree}.

For the final claim,  because  of  \upitem2 and \lref{lem2145}\upitem2 we have \( \rank[f_{xxx}]\geq 3 \). On the other hand,  by \upitem3, \(f_{xxx}\) is a sum of three cubes and so   \(\rank[f_{xxx}]\leq 3\).
\end{proof}

It is frequently claimed that \( S=0 \) is a necessary and sufficient condition for \( f_{xxx} \) to be written as a sum of three cubes \cite{ottaviani2007} .  It is clearly a necessary condition because of \lref{lem227}\eqref{iithree}, but it is not sufficient as seen in Lemma~\ref{lem35} and \ref{lem34}. Specifically, the forms \(f_{xxx}=  a_x^2 c^{\phantom{2}}_x +  b_x^3\) and \(f_{xxx}=a_x(a_x c_x+ b_x^2)  \)  with linearly independent  forms  \(\{a_x,b_x,c_x\}\), have  ranks \( 5 \) and \( 4 \) so cannot be written as sums of three cubes, even though \( S=T=0 \) for these forms.

\begin{method}\label{meth4}
Suppose that \( f_{xxx} \) is a cubic form such that \( S=0 \) and \( T\neq 0 \). Write \( \Delta_{xxx}=a_xb_xc_x \) for some linearly independent  forms \( \{a_x,b_x,c_x\} \). \textup(Factoring methods are discussed in \cite{brookfield2021}.\textup)  Then  \(f_{xxx}=a_0a_x^3+b_0b_x^3+c_0c_x^3 \) for some \( a_0,b_0,c_0\in \C \) that can be determined by matching coefficients.
\end{method}

If, as in the above lemma, \( S=0 \) and \( T\neq 0 \) for a cubic form \( f_{xxx} \), then \lref{lem227}\eqref{iithree} shows that the  forms appearing in the expression \(f_{xxx}=a_x^3+b_x^3+c_x^3 \) are the linear factors of \( \Delta_{xxx} \). Because of unique factorization, they are unique, up to order and multiplication by scalars.

\begin{example}\label{ex39}
Suppose that \( f_{xxx}=x_2^2 x_3-x_1^3-q x_3^3 \) for some nonzero \( q\in \C \). Then 
\begin{gather*}
\btheta_{uuxx}=12 q x_3 (3 u_2^2 x_1 - u_1^2 x_3) + 
4 (6 u_2 u_3 x_1 x_2 - u_1^2 x_2^2 - 3 u_3^2 x_1 x_3)\\
 \Delta_{xxx}=12 x_1(x_2^2-3q x_3^2) \qquad S=0 \qquad T=864 q 
\end{gather*}
 Since \( q\neq 0 \),  \lref{lem36} applies and \( \rank[f_{xxx}]=3 \). To express \( f_{xxx} \) as a sum of three cubes, we factor \( \Delta_{xxx} \): \[ \Delta_{xxx}=12  x_1 (x_2 +  \tau x_3) (x_2 - \tau x_3) \] where \( \tau^2+3q=0 \), and we expect \( f_{xxx}=a_0 x_1^3+b_0(x_2 + \tau x_3)^3+ c_0 (x_2 - \tau x_3)^3 \) for some \( a_0,b_0,c_0\in \C \). By matching coefficients we get 
 \begin{equation}\label{eq1314}
 f_{xxx} =-  x_1^3 + \frac1{6 \tau} (x_2 + \tau x_3)^3 - \frac1{6 \tau} (x_2 - \tau x_3)^3 
\end{equation}
\end{example}

\begin{lemma}\label{lem37}
For a cubic form \(f_{xxx}\), the following are equivalent:
\begin{enumerate}
\item  \( S\neq 0 \)

\item \(   f_{xxx}=a_x^3+b_x^3+c_x^3+d_x^3  \) for some triplewise independent  forms \( \{a_x,b_x,c_x,d_x\} \)
\end{enumerate}
When these conditions hold,  \(\rank[f_{xxx}]=4\).
\end{lemma}

\begin{proof}
(1)\( \Rightarrow \)(2): First we show that if \( S\neq 0\), then \(f_{xxx}\) is a sum of four cubes. To do this, we seek a  form \( u_x\) and \(u_0\in \C\) such that  \(g_{xxx}=f_{xxx}-u_0 u_x^3\) is a sum of three cubes. The \(S\) and \(T\) concomitants of \(g_{xxx}\) are 
\begin{equation}\label{eq2134}
 \Vert S \Vert_{f\mapsto g} = S - 24 u_0 S_{uuu} \qquad 
 \Vert T \Vert_{f\mapsto g}= T - 36 u_0 T_{uuu} + 216 u_0^2  F_{6u} 
\end{equation}

\emph{Caution: Here we are treating the coefficients of \( u_x \) \textup(that is \( u_1 \), \( u_2 \) and \( u_3 \)\textup) as constants rather than variables. So \( S_{uuu} \), \( F_{6u} \), etc. are also constants rather than forms in \( \uu \).}

Because of \lref{lem36}, it would suffice to find \( u_x\) and \(u_0\in \C\)  so that \(\Vert S \Vert_{f\mapsto g}=0\) and  \(\Vert T \Vert_{f\mapsto g}\neq 0\). The first of these goals is easily accomplished. Simultaneously obtaining the second goal is trickier and depends on the identity 
\begin{equation}\label{eq2137}
72  u_0^2  \bF_{6u}  = (S T + 24 u_0 TS_{uuu}  -  36  u_0 S T_{uuu}) \Vert S \Vert_{f\mapsto g} - S^2  \Vert T \Vert_{f\mapsto g} 
\end{equation}
derived from \eqref{eq2145} and \eqref{eq2134}. 

\new Since \(S\neq 0\), \eqref{eq34} implies that \(\bF_{6u} \) and   \(S_{uuu}\) are nonzero. 
Now  fix \(\uu\in \C\) so that \(S_{uuu}\bF_{6u} \in \C\) is nonzero. In particular,  \(S_{uuu}\) and \(\bF_{6u} \) are both nonzero for that choice of \(\uu\in \C\).  Now fix \(u_0\in \C\) so that \(  S - 24 u_0 S_{uuu} =0\), that is,  \(\Vert S \Vert_{f\mapsto g} =0\). Since \(S\neq 0\), we have \(u_0\neq 0 \). With this choice of \(\uu\) and \(u_0\), \eqref{eq2137} becomes \( 72  u_0^2  \bF_{6u}  = - S^2  \Vert T \Vert_{f\mapsto g}\). Thus \(\Vert T \Vert_{f\mapsto g}\) is nonzero. 

Finally,  because of \lref{lem36}, \(g_{xxx}\) is a sum of three cubes, and consequently, \(f_{xxx}=g_{xxx}+u_0u_x^3\) is a sum of four cubes, that is, \(   f_{xxx}=a_x^3+b_x^3+c_x^3+d_x^3  \) for some  forms \( \{a_x,b_x,c_x,d_x\} \).

With  \(   f_{xxx}=a_x^3+b_x^3+c_x^3+d_x^3  \) and \( S\neq 0 \), \lref{lem227}\eqref{iifour} implies that \( [abc] \), \( [abd] \), \( [acd] \) and \( [bcd] \) are all nonzero, and so any triple of these four forms is independent.

(2)\( \Rightarrow \)(1):  This is immediate from \lref{lem227}\eqref{iifour}.

For the final claim, \( f_{xxx}  \) is a sum of four cubes, so \( \rank[f_{xxx}]\leq 4 \). On the other hand, since  \( S\neq 0 \),  \lref{lem2145}\upitem4 implies that the  rank of \(f_{xxx} \) is at least four. 
\end{proof}

The proof of this lemma provides a method of expressing \( f_{xxx} \) as a sum of four cubes when \( S\neq 0 \):

\begin{method}\label{meth5}
Suppose that \( f_{xxx} \) is a cubic form with \( S\neq 0 \). Fix \( (\uu)\in \C^3 \) so that  \(S_{uuu}\) and \(\bF_{6u} \) are  nonzero. Fix \(u_0\in \C\) so that \(  S - 24 u_0 S_{uuu} =0\), and set \(g_{xxx}=f_{xxx}-u_0 u_x^3\). Then the Hessian of \( g_{xxx} \) is completely reducible so can be written as \( a_xb_xc_x \) for some  forms \( a_x \), \( b_x \) and \( c_x \). \textup(Factoring methods are discussed in \cite{brookfield2021}.\textup) Then \( f_{xxx}=a_0a_x^3+b_0 b_x^3+c_0 c_x^3+u_0 u_x^3 \) for some \( a_0,b_0,c_0\in \C \) that can be determined by matching coefficients in this equation. 
\end{method}

\begin{example}
Suppose that \( f_{xxx}=x_1 (x_2^2 + x_3^2) \). Then \( S= 16\), so, by \lref{lem37}, \( \rank[f_{xxx}]=4 \). To express \( f_{xxx} \) as a sum of four cubes we follow \mref{meth5}. Since
\[ S_{uuu}=4 u_1 (u_2^2 + u_3^2) \quad \bF_{6u} =-1024 u_1^2 (u_2^2 + u_3^2)^2, \]  we can choose, for example, \( u_x=x_1+x_2 \), that is, \( u_1=u_2=1 \) and \( u_3=0 \), so that \( S_{uuu} \) and \(  \bF_{6u} \) are nonzero.  Setting \(g_{xxx}=f_{xxx}-u_0 u_x^3\), we find \( \Vert S\Vert_{f\mapsto g}=16 (1 - 6 u_0) \), so we pick \( u_0=1/6 \) so that \( \Vert S\Vert_{f\mapsto g} =0\). At the same time \( \Vert T\Vert_{f\mapsto g} =8\) is nonzero, so \( \rank[g_{xxx}]=3 \),  as expected.

To express \( g_{xxx} \) as a sum of three cubes we use \mref{meth4}. The Hessian of \( g_{xxx} \) is
\[ \Vert \Delta_{xxx}\Vert_{f\mapsto g}=2 (x_2 - x_1) (x_1^2 + x_3^2)=2 (x_2 - x_1) (x_1+ i x_3) (x_1- i x_3) \] and \( g_{xxx} \) is a linear combination of the cubes of the linear factors of  its Hessian, with coefficients that can be determined by matching coefficients. Once that is done, \( f_{xxx}=g_{xxx}+u_0u_x^3 \) can be written as a sum of  four cubes. The final result is 
\[ 6f_{xxx}=(x_1 + x_2)^3 + (x_1 - x_2)^3 - (x_1 + i x_3)^3 - (x_1 -i x_3)^3 \] 

Since \( f_{xxx}=x_1 (x_2 + i x_3)  (x_2 - i x_3) \) is completely reducible, other ways of expressing \( f_{xxx} \) as a sum of cubes are discussed in the following section.
 
 \end{example}

\section{Completely Reducible Forms}\label{sectreducible}

In this, and the following sections, we calculate the ranks of cubic forms in some special cases.

Suppose that  \(f_{xxx}\) is completely reducible, that is, \( f_{xxx}=a_xb_xc_x \) for some  forms \( a_x \), \( b_x \), \( c_x \). Then, extending \lref{lem22}\eqref{iiabc},
\begin{equation}\label{eq2117}
\begin{aligned}
\Delta_{xxx}&=[abc]^2\, a_xb_xc_x & 
S&=[abc] ^4 &
F_{6u}&=[abu]^2 [bcu]^2 [cau]^2\\
S_{uuu}&=[abc] [abu][bcu] [cau]&
T&=[abc] ^6 &
 \bF_{6u}&=[abc]^8 [abu]^2 [bcu]^2 [cau]^2
\end{aligned}
\end{equation}

\new If  \( S=0 \), then \( \Delta_{xxx}=0 \),  \( [abc]=0 \) and  so  \(\{a_x,b_x,c_x\}\) is dependent. The  rank of \( f_{xxx} \) is \( 1 \), \( 2 \) or \( 3 \) depending on the values of \( \theta_{uuxx} \) and \( F_{6u} \). These cases are discussed in \eref{ex49}, \mref{meth2} and \lref{lem31}. 

\new If \( S\neq 0 \), then \( [abc]\neq 0 \) so \( \{a_x,b_x,c_x\} \) is linearly independent, and, by \lref{lem37},  \(\rank[f_{xxx}]=4 \). Here's one way \( f_{xxx} \) can be expressed as a sum of four cubes:
\begin{equation}\label{eq128}
24 a_xb_xc_x =(a_x + b_x + c_x)^3 - (a_x + b_x - c_x)^3 - (b_x + c_x- a_x)^3 - (c_x + a_x - b_x)^3 
\end{equation}
\mref{meth5} provides infinitely many other ways of expressing \( f_{xxx} \) as a sum of four cubes. Indeed, one of the terms  could be \( u_x^3 \) where  \( u_x \) is any form whose coefficients \( \uu \) satisfy   \(S_{uuu}\bF_{6u}\neq 0\). 

Let's look for all such \( u_x \). 
Since \( \{a_x,b_x,c_x\} \) is linearly independent, \( u_x \) can  conveniently be written as  \( u_x =a_0a_x+b_0b_x+c_0c_x \) with \( a_0,b_0,c_0\in \C \). In this circumstance, using \eqref{eq2117}, 
\[ S_{uuu}=a_0b_0c_0 [abc]^4 \qquad \bF_{6u}=a_0^2b_0^2c_0^2 [abc]^{14}  \]  
If \( a_0b_0c_0 \neq 0 \), then \(S_{uuu}\bF_{6u}\) is nonzero and, by  \lref{lem37}, \( u_x^3 \) as a summand in an expression of \( f_{xxx} \) as a sum of four cubes. This new expression for \( f_{xxx} \) is easily obtained from  \eqref{eq128} by replacing \( a_x \) by \( a_0a_x \),  \( b_x \) by \( b_0b_x \), and \( c_x \) by \( c_0c_x \):
\begin{multline}\label{eq129}
24 a_0b_0c_0 a_xb_xc_x =(a_0a_x + b_0b_x + c_0 c_x)^3 - (a_0a_x + b_0b_x - c_0 c_x)^3 \\
- (b_0b_x + c_0 c_x- a_0a_x)^3 - (c_0 c_x + a_0a_x - b_0b_x)^3 
\end{multline}
 Since  \( a_0b_0c_0 \neq 0 \),  this equation can be solved for \( f_{xxx} \),  expressing \( f_{xxx} \) as a sum of four cubes in which the first term is the cube of \( a_0a_x+b_0b_x+c_0c_x \). 
 
\section{Conic and Secant Line}

Suppose that \( f_{xxx} \) is reducible but not completely reducible, that is, \(f_{xxx}= a_xb_{xx} \) where \( b_{xx} \) is an irreducible quadratic form and \( a_x \) is a linear form. Geometrically, the curve \( f_{xxx}=0 \) is the union of the conic \( b_{xx}=0 \) and the line \( a_x=0 \). We know already that, if the line \( a_x=0 \) is tangent to the conic \( b_{xx}=0 \), then \lref{lem34} applies and \( \rank[f_{xxx}]=5 \). 

In this section, we consider the remaining case where  \( a_x=0 \) is a secant line of the  conic \( b_{xx}=0 \). From Lemmas~\ref{lemx66} and ~\ref{lemx77}, \( f_{xxx} \) can be written as  \( f_{xxx}=a_x (a_x^2 + b_x c_x) \) for some  linearly independent  forms \( \{a_x,b_x,c_x\} \).  In this circumstance,  \( S=[abc]^4 \). Because \( \{a_x,b_x,c_x\} \) is independent, \( S \)  is nonzero and by \lref{lem37},   \( \rank[f_{xxx}]=4 \).

To express \( f_{xxx} \) as a sum of four cubes, following \mref{meth5}, we look for a form  \( u_x \) and \( u_0\in \C \) such that \( \rank[f_{xxx}-u_0 u_x^3]=3 \). Since \( \{a_x,b_x,c_x\} \) is linearly independent, \( u_x \) can  conveniently be written as  \( u_x =a_0a_x+b_0b_x+c_0c_x \) for some \( a_0,b_0,c_0\in \C \).  With this choice for \( u_x \) we get 
\[ S_{uuu}=a_0b_0c_0 [abc]^4 \qquad \bF_{6u}=b_0^2 c_0^2 (a_0^2 - 12 b_0 c_0 ) [abc]^{14}  \] 
To ensure that \( S_{uuu} \) and \( \bF_{6u} \) are both nonzero, we will assume that \( a_0 \), \( b_0 \), \( c_0 \) and \(  a_0^2 - 12 b_0 c_0 \) are all nonzero.  We set \( u_0=1/(24 a_0 b_0 c_0) \)  so that  \(  S - 24 u_0 S_{uuu} =0\), and then
\[ g_{xxx}=f_{xxx}-u_0 u_x^3=f_{xxx} - \frac1{24 a_0 b_0 c_0}(a_0a_x+b_0b_x+c_0c_x)^3  \]  Then \( \Vert S \Vert_{f\mapsto g} =0\) and   \( 8a_0^2 \Vert T \Vert_{f\mapsto g} =- [abc]^6 (a_0^2 - 12 b_0 c_0 )\). Since \( \Vert T \Vert_{f\mapsto g}  \) is nonzero,  \(g_{xxx} \) has  rank three. To express \( g_{xxx} \) as a sum of three cubes we need to factor the Hessian of \( g_{xxx} \):
\begin{align*}
&\Vert \Delta_{xxx}\Vert_{f\mapsto g} \\
&=\frac1{8 a_0 b_0 c_0} [abc]^2 (a_0 a_x - b_0 b_x - 
   c_0 c_x) \left((a_0^2 - 12 b_0 c_0 ) a_x^2 - (b_0 b_x - c_0 c_x)^2\right)\\
   &=\frac1{8 a_0 b_0 c_0}  [abc]^2 (a_0 a_x - b_0 b_x - 
   c_0 c_x) (\sigma a_x - (b_0 b_x - c_0 c_x)) (\sigma a_x + (b_0 b_x - c_0 c_x))
\end{align*}
where \( \sigma^2=a_0^2 - 12 b_0 c_0 \). Hence \( g_{xxx} \) is a linear combination of the cubes of the three linear factors of \( \Vert \Delta_{xxx}\Vert_{f\mapsto g}  \). By matching coefficients we can express \( g_{xxx} \) as a sum of three cubes and then \( f_{xxx} \) as a sum of four cubes. The final result for \( f_{xxx} \) is 
\begin{align*}
24 a_0 b_0 c_0 \sigma  f_{xxx} &=\sigma  (a_0 a_x + b_0 b_x + c_0 c_x)^3 + 
   \sigma  (a_0 a_x - b_0 b_x - c_0 c_x)^3\\
   &\qquad -  a_0 (\sigma  a_x - (b_0 b_x - c_0 c_x))^3 -  
   a_0 (\sigma  a_x + (b_0 b_x - c_0 c_x))^3
\end{align*}
This formula gives infinitely many ways of expressing \( f_{xxx} \) as a sum of four cubes. It holds so long as  \( a_0 b_0 c_0 \sigma \neq 0 \) and \( \sigma^2=a_0^2 - 12 b_0 c_0 \). If only one way is needed, we can avoid square roots by setting  \( a_0=4 \), \( b_0=c_0=1 \) and \( \sigma=2 \) to get 
\[ 96f_{xxx}=  (4 a_x + b_x + c_x)^3 +  (4 a_x - b_x - c_x)^3 -  2 (2 a_x - b_x + c_x)^3 - 2 (2 a_x + b_x - c_x)^3 \] 

\section{Hesse Normal Form}

Suppose that \( f_{xxx}= s (a_x^3 + b_x^3 + c_x^3) + t a_x b_x c_x\) for some  linearly independent  forms \( \{a_x,b_x,c_x\} \) and \( s,t\in \C \), not both zero. 
Then
\begin{align*}
S&=t (t^3 - 216 s^3)\,[abc]^4 =t ( t-6 s ) (t^2+ 6 s t+36 s^2 ) [abc]^4 \\
T&= (t^6 + 54_0 s^3 t^3 - 5832 s^6)\, [abc]^6
\end{align*}
Since \( [abc]\neq 0 \),  it is not possible for \( S \) and \( T \) to be simultaneously zero. So, from the table, the  rank of \( f_{xxx} \) is \( 3 \) or \( 4 \) depending on whether \( S \) is zero or not.

In view of the symmetry in this problem, we apply \mref{meth5} with the naive assumption that \( u_x=a_x+b_x+c_x \). With \( u_x \) fixed this way we get
\[ S_{uuu}=(6 s - t) (3 s + t)^2\,[abc]^4\qquad 
\bF_{6u}= (6 s - t)^8 (3 s + t)^4\,[abc]^{14}
 \]
We don't expect our method to work unless both \( S_{uuu} \) and \( \bF_{6u} \) are nonzero, so we assume that \( (6 s - t) (3 s + t)\neq 0 \). Next we solve \(  S - 24 u_0 S_{uuu} =0\) to get 
\[ u_0=\frac{t \left(36 s^2+6 s t+t^2\right)}{24 (3 s+t)^2} \]
Setting \( g_{xxx}=f_{xxx}-u_0u_x^3 \) we find 
\begin{align*}
24 (3 s + t)^2 g_{xxx}&= (216 s^3 + 108 s^2 t + 18 s t^2 - t^3) (a_x^3 + b_x^3 + c_x^3) \\
&\quad - 3 t (36 s^2 + 6 s t + t^2) (a_x^2 b_x + a_x b_x^2 + a_x c_x^2 + c_x a_x^2 + b_x c_x^2 + c_x b_x^2) \\
&\quad+ 18 t^2 (6 s + t) a_x b_x c_x
\end{align*}
As expected, the Hessian of \( g_{xxx} \) factors completely:
\[ 8 (3 s + t)^2\Vert \Delta_{xxx}\Vert_{f\mapsto g} =[abc]^2 (6 s - t)^2 A_xB_xC_x\]
where
\begin{gather*}
A_x=(6 s + t) a_x - t (b_x + c_x) \qquad 
B_x=(6 s + t) b_x - t (c_x + a_x)\\
C_x=(6 s + t) c_x - t (a_x + b_x)
\end{gather*}
We also expect that \( f_{xxx} \) is a linear combination of \( A_x^3 \), \( B_x^3 \), \( C_x^3 \) and \( u_x^3 \) with scalars that can be determined by matching coefficients. After some computation, the final result is the identity
\begin{equation}\label{eq126}
24 (3 s + t)^2 f_{xxx}=  A_x^3 + B_x^3 +C_x^3 +t (36 s^2 + 6 s t + t^2) u_x^3.
\end{equation}
Even though this identity was derived assuming that \( (6 s - t) (3 s + t)\neq 0 \), it is valid independent of that assumption and is useful so long as \( 3 s + t\neq 0 \).

With \eqref{eq126} established, we start over again and consider the  cases \( S\neq 0 \) and \( S=0 \).

\new Suppose that \( S\neq 0 \) and hence \( \rank f_{xxx}=4 \). We have two cases:
\begin{itemize}
\item If \( 3 s + t \) is nonzero, then \eqref{eq126} can be solved for \( f_{xxx} \) expressing it as a sum of four cubes. Note that the coefficient of \( u_x^3 \) is nonzero since \( S\neq 0 \).

\item If  \( 3 s + t =0\), then \( S=9 [abc]^4  t^4\neq 0\) and \( f_{xxx} \) is completely reducible:
\[ f_{xxx}= s(a_x + b_x + c_x) (a_x + \omega b_x + \omega^2 c_x)(a_x + \omega^2 b_x +  \omega c_x) \]
 where \( \omega =e^{2\pi i/3}\). So \eqref{eq128} can be used to express \( f_{xxx} \) as a sum of four cubes:
\begin{multline*}
\qquad\qquad 24 f_{xxx} =s \Big(27  a_x^3 - (a_x - 2 b_x - 2 c_x)^3 
- (a_x - 2 \omega b_x -  2 \omega^2 c_x)^3\\ - (a_x - 2 \omega^2 b_x - 2 \omega c_x)^3\Big)
\end{multline*}
\end{itemize}
 
\new Suppose that \( S=0 \) and hence \( \rank f_{xxx}=3 \). Since \( t ( t-6 s ) (t^2+ 6 s t+36 s^2 )=0 \), we have two cases.

\begin{itemize}
\item If \( t (t^2+ 6 s t+36 s^2 ) =0\), then  the last term in \eqref{eq126} is zero, and since \( 3 s + t \) cannot be zero (or else \( s=t=0 \)),  \eqref{eq126} can be used to express \( f_{xxx} \) as a sum of three cubes.

\item The only remaining case is \( 6s-t=0 \). Here \( \Delta_{xxx} \) is completely reducible:
\[ \qquad \Delta_{xxx} =-
 108 [abc]^2 s^3 ( a_x + b_x + c_x) ( a_x + \omega b_x + \omega^2 c_x) ( 
   a_x + \omega^2 b_x + \omega c_x) \]
 As explained in \mref{meth4}, the linear factors of \( \Delta_{xxx} \) occur in the expansion of \( f_{xxx} \) as a sum of three cubes: 
\[ \qquad 3 f_{xxx} = s \Big((a_x + b_x + c_x)^3 + (a_x + \omega b_x + \omega^2 c_x)^3 + (a_x + \omega^2 b_x +  \omega c_x)^3\Big) \] 

\end{itemize}

Notice that \eqref{eq126} can be used to express \( f_{xxx} \) as a sum of cubes  except when \( 3s+t=0 \) or \( 6s-t=0 \).  In these special cases,  the appearance of \( \omega \) in such expressions  is unavoidable.

\section{Weierstrass Form}
Suppose that \( f_{xxx}=x_2^2 x_3 - ( x_1^3 + p x_1 x_3^2 + q x_3^3)\) for some \(p,q\in \C\). Then \(S=-48p\) and \(T=864 q\).  If \(p\neq 0\), then \(S\neq0 \) and the  rank of \(f_{xxx}\) is four. To express \(f_{xxx}\) as a sum of four cubes using \mref{meth5} we look for \(\uu\in \C\) so that 
\[ S_{uuu}=2 p u_1^3 + 18 q u_1 u_2^2 - 12 p u_2^2 u_3 + 6 u_1 u_3^2 \] 
and  \(\bF_{6u} \) are nonzero. 

The obvious choice to make \( S_{uuu} \) nonzero is \(u_1=1\), \(u_2=u_3=0\) since then \( S_{uuu}=2p\neq 0 \).  Unfortunately, this choice of \( \uu \) gives  \(\bF_{6u} =0\), so \mref{meth5} cannot be used to express \( f_{xxx}  \) as a sum of four cubes.

 It is worth noticing that if we make this ``bad'' choice for \( \uu \), then \( u_0=-1 \) so that \(g_{xxx}=f_{xxx}-u_0 u_x^3= x_3 (x_2^2 - (p x_1  +q x_3)x_3) \) satisfies \( \Vert S \Vert_{f\mapsto g} =0\). Unfortunately,   \( \Vert T \Vert_{f\mapsto g} =0\), so \lref{lem36} does not apply, and there is no certainty that \( g_{xxx} \) can be written as a sum of three cubes. Indeed, \( g_{xxx} \) has the form of \lref{lem34}(3), so \( \rank[g_{xxx}]=5 \). This illustrates why the condition \( \bF_{6u}\neq 0 \) is necessary in the proof of \lref{lem37}. 

Starting over again,  we fix \( u_x=x_2+\tau x_3  \)  (that is, \(u_1 = 0\),  \(u_2 =1\)  and \(u_3 = \tau \)),  with \( \tau^2+3q=0 \), because then \(S_{uuu}= -12 p \tau\) and  \(\bF_{6u} =27648 p^5\). These quantities are both nonzero so long as we add the extra assumption that \( q \) is nonzero. 

We set \( g_{xxx}=6 \tau f_{xxx} -  u_x^3 \) so that  \( \Vert S \Vert_{f\mapsto g} =0\) and  \( \Vert T \Vert_{f\mapsto g} =-6^9 p^3 q^2\neq 0\). By \lref{lem37}, \( \rank[g_{xxx}]=3 \). The Hessian of \( g_{xxx} \) is  \[ \Vert \Delta_{xxx}\Vert_{f\mapsto g}=-6^4 p q (x_2 - \tau x_3) (p x_3^2 - 3 x_1^2)= 432  p q (x_2 - \tau x_3) (3 x_1 + \sigma x_3) (3 x_1 - \sigma x_3) \] where \( \sigma^2+3p=0 \). By matching coefficients, it is easy to express \( g_{xxx} \) as a sum of the cubes of the factors of its Hessian:
\[ 9 g_{xxx} = -9 (x_2 - \tau x_3)^3 - \tau  (3 x_1 + \sigma x_3)^3 - 
   \tau  (3 x_1 - \sigma x_3)^3 \] 
 Using \(6\tau  f_{xxx}=g_{xxx}+ u_x^3 \), we can express \( f_{xxx} \) as a sum of four cubes: 
  \begin{equation}\label{eq51}
   54 \tau f_{xxx} =9 (x_2 + \tau x_3)^3 - 9 (x_2 - \tau x_3)^3 - 
   \tau  (3 x_1 + \sigma x_3)^3 - \tau  (3 x_1 - \sigma x_3)^3
 \end{equation}
 
\new In retrospect, we notice that this equation follows from the  two equations
\begin{align}
6 \tau x_3 (x_2^2 - q x_3^2)&=(x_2 + \tau x_3)^3 - (x_2 - \tau x_3)^3  \label{eq55}\\
54 x_1 (x_1^2 + p x_3^2)&=(3 x_1 + \sigma x_3)^3 + (3 x_1 - \sigma x_3)^3 \label{eq56}
\end{align}
(compare \eref{ex36}) and that \( f_{xxx} \) is the sum of two terms 
\begin{equation}\label{eq57}
f_{xxx}=x_3 (x_2^2 - q x_3^2) - x_1 (x_1^2 + p x_3^2)
\end{equation}
each having  rank two. 

\new So far we have assumed that \( p \) and \( q \) are nonzero. What about the other cases? The case  \( p=0 \) and \( q\neq 0 \) is discussed in \eref{ex39}. We notice that, in this case,  \eqref{eq55} and \eqref{eq57} can be combined to express \( f_{xxx} \) as a sum of three cubes, exactly as in \eqref{eq1314}.

If \( p\neq 0 \) and \( q= 0 \), then \( S\neq 0\) and \( T= 0 \), so \( \rank[f_{xxx}]=4 \). The above argument does not apply because it assumed \( q\neq 0 \). We could start again with a different choice for \( \uu \), but instead, motivated by \eqref{eq57}, we seek to write \( f_{xxx}=g_{xxx}-h_{xxx} \) with \( \rank[g_{xxx}]=\rank[h_{xxx}]=2 \). We have 
\begin{equation}
f_{xxx}=x_3 x_2^2  - x_1 (x_1^2 + p x_3^2).
\end{equation}
Choosing \( g_{xxx}=x_3 x_2^2 \) and \( h_{xxx}= x_1 (x_1^2 + p x_3^2) \) will not help because,  by \lref{lem32}, \( \rank[g_{xxx}]=3 \). Instead we choose \( g_{xxx}=x_3 x_2^2 + a x_3^3\) and \( h_{xxx}= x_1 (x_1^2 + p x_3^2)+a x_3^3\) with \( a\in \C \) to be determined. This choice means that each of \( g_{xxx} \) and \( h_{xxx} \) are binary forms. As a consequence their Hessians are zero \cite[Lemma 8.5]{brookfield2021} which,  by \lref{lem33},  is a necessary condition for these forms to have rank two. 

The \( F_{6u} \) concomitants for  \( g_{xxx} \) and \( h_{xxx} \) are \( -4 a u_1^6 \) and  \( -(27 a^2 + 4 p^3) u_2^6 \) respectively. By \lref{lem33}, if \( a\neq 0 \) and \( 27 a^2 + 4 p^3\neq0 \), then  \( g_{xxx} \) and \( h_{xxx} \) will have  rank two. Following \mref{meth23}, the factorization of the  \( \theta_{uuxx} \) concomitant of \( g_{xxx} \), namely \(  \Vert \theta_{xxuu}\Vert_{f\mapsto g} \), can be used to express \( g_{xxx} \) as a sum of two cubes:
\begin{align*}
 \Vert \theta_{xxuu}\Vert_{f\mapsto g} &= -4 u_1^2  (x_2 + \mu x_3)(x_2 - \mu x_3)\\
6 \mu g_{xxx} &= (x_2 + \mu x_3)^3-(x_2 - \mu x_3)^3 
\end{align*}
where \( \mu^2= 3 a\neq 0 \). Similarly, the factorization of  \(  \Vert \theta_{xxuu}\Vert_{f\mapsto h} \), can be used to express \( h_{xxx} \) as a sum of two cubes:
\begin{align*}
3 p \Vert \theta_{xxuu}\Vert_{f\mapsto h} &= 
 u_2^2 \big(6 p x_1 + (9 a  - \nu ) x_3\big) \big(6 p x_1 + (9 a + \nu ) x_3\big)\\
 432   p^3 \nu \,  h_{xxx} &= ( 
      \nu +9 a ) \big(6 p x_1 + (9 a  - \nu  )x_3\big)^3 + (
      \nu-9 a  ) \big(6 p x_1 + (9 a + \nu ) x_3\big)^3
\end{align*}
where \(\nu^2= 3 (27 a^2 + 4 p^3 ) \neq 0\). Using these expressions for \( g_{xxx} \) and \( h_{xxx} \), \( f_{xxx}=g_{xxx}-h_{xxx} \) can be expressed as a sum of four cubes, as claimed: 
\begin{multline*}
f_{xxx}= \frac{1}{6 \mu} (x_2 + \mu x_3)^3 
- \frac{1}{6 \mu} (x_2 - \mu x_3)^3 \\
- \frac{ \nu+9 a }{432   p^3 \nu} (6 p x_1 + (9 a - \nu )x_3)^3
 - \frac{ \nu -9 a}{432   p^3 \nu} (6 p x_1 + (9 a  + \nu) x_3)^3
\end{multline*}

\end{document}